\documentclass[a4paper,11pt]{amsart}
\usepackage{amsmath,amssymb,mathabx,amsfonts,graphicx}
\usepackage[mathscr]{euscript}
\usepackage[all]{xy}
\usepackage{enumitem}
\usepackage{inputenc}
\usepackage{xcolor}
\setenumerate[1]{label={\textup{(\arabic*)}}}
\numberwithin{equation}{section}
    \theoremstyle{plain}
\newtheorem{theorem}{Theorem}[section]

\newtheorem{lemma}[theorem]{Lemma}
\newtheorem{proposition}[theorem]{Proposition}
\newtheorem{corollary}[theorem]{Corollary}

\newtheorem{definition}[theorem]{Definition}

\newtheorem{example}[theorem]{Example}
\theoremstyle{remark}

\theoremstyle{definition}
\newtheorem{remark}[theorem]{Remark}

\newcommand\<[1]{\left\langle\, #1\,\right\rangle}

\newcommand\norm[1]{\lVert #1 \rVert}

\newcommand{\cf}{\ensuremath \mbox{ ${\mathbf{1}}$}}
\newcommand{\Cf}{\ensuremath \mbox{\large${\mathbf{1}}$}}

\newcommand{\h}{\ensuremath \mathbb{H}}

\newcommand{\bd}{{\ast\ast}}

\newcommand{\G}{\ensuremath \mbox{\tiny $G$}}

\newcommand{\dmg}{\ensuremath  \, \mathrm{dm}_{\G}}

\newcommand{\mg}{\ensuremath  \, \mathrm{m}_{\G}}

\newcommand{\mA}{\mathscr{A}}

\newcommand\restr[2]{{
  \left.\kern-\nulldelimiterspace 
  #1 
  \vphantom{|} 
  \right|_{_{#2}} 
  }}

\newcommand{\C}{\mathbb{C}}

\newcommand{\T}{\mathbb{T}}

\newcommand{\N}{\mathbb{N}}

\DeclareMathOperator{\supp}{supp}

\renewcommand{\ll}{\mathcal{L}}

\renewcommand{\emptyset}{\varnothing}

\DeclareMathOperator{\VN}{VN}

\title[Ergodicity of positive definite functions]{
Positive definite functions as uniformly ergodic  multipliers of the Fourier algebra}
\author{Jorge Galindo,  Enrique Jord\'a \and Alberto Rodr\'iguez-Arenas}
\address{\noindent Jorge Galindo, Instituto Universitario de Matem\'aticas y
Aplicaciones (IMAC)\\ Universidad Jaume I, E-12071, Cas\-tell\'on,
Spain. \hfill\break \noindent E-mail: {\tt jgalindo@uji.es}}
\address{\noindent Enrique Jord\'a, EPS Alcoy, Instituto Universitario de Matem\'atica Pura y Aplicada IUMPA,
Universitat Polit\`ecnica de Val\`encia, Plaza Ferr\'andiz y Carbonell s/n
E-03801 Alcoy, Spain, \hfill\break \noindent E-mail: {\tt ejorda@mat.upv.es}}
\keywords{ergodic measure, uniformly ergodic measure, random walk, mean ergodic operator, uniformly mean ergodic operator, convolution operator, locally compact  group, measure algebra}
\address{\noindent Alberto Rodr\'{\i}guez-Arenas \\Departamento de Álgebra, Análisis Matemático, Geometría y Topología, Facultad de Ciencias,Universidad de Valladolid, Paseo de Belén 7, E-47011 Valladolid, Spain
\hfill\break \noindent E-mail: {\tt arodare@uva.es}}

\date{\today.}

\begin{document}\maketitle

\begin{abstract}%

      Let G be a locally compact group and let $\phi$ be a positive definite function on G with $\phi(e)=1$. This function defines a multiplication operator $M_\phi$ on the Fourier algebra $A(G)$ of $G$.  The aim of this paper is to classify the ergodic properties of the operators $M_\phi$, focusing on several key factors, including the subgroup $H_\phi=\{x\in G\colon \phi(x)=1\}$, the spectrum of $M_\phi$, or how ``spread-out'' a power of $M_\phi$ can be. We show that the multiplication operator $M_\phi$ is uniformly mean ergodic if and only if $H_\phi$ is open and 1 is not an accumulation point of the spectrum of $M_\phi$.    Equivalently, this happens when some power of $\phi$ is not far, in the multiplier norm, from a function supported on finitely many cosets of $H_\phi$. Additionally, we show that the powers of  $M_\phi$ converge in norm if, and only if, the operator is uniformly mean ergodic and  $H_\phi =\{x\in G\colon |\phi(x)|=1\}$.
      \end{abstract}

\section{Introduction}
The ergodicity of the  random walk  governed by a probability measure $\mu$ on a group $G$ can be described through the behaviour of the  $L^1(G)$-convolution operator, $\lambda_1(\mu)f=\mu \ast f$,  restricted  to the augmentation ideal $L_1^0(G)=\{f\in L^1(G)\colon \int f(x)\dmg(x)=0\} $. When $G$ is commutative, the properties of $\lambda_1(\mu)$ can be  recast, via the Fourier-Stieltjes transform,  in terms of the multiplication operators $M_{\widehat{\mu}} \colon A(\widehat{G})\to A(\widehat{G})$, where $A(\widehat{G})$ denotes the Fourier algebra of $\widehat{G}$, the algebra of  functions on $\widehat{G}$ that can be obtained as the Fourier transform of some function in $L_1(G)$.

While the Fourier-Stieltjes transform can hardly be given sense beyond commutative or compact groups, the Fourier algebra can still be meaningfully defined as an algebra of functions on any locally compact group. Harmonic Analysis  on commutative groups, which is so often developed with the aid of   transforms, admits  then two noncommutative generalizations, one through    convolution operators on the noncommutative algebra $L_1(G)$ and another one leaning  on multiplication operators on the Fourier algebra $A(G)$.

In our previous works \cite{galijorda21,gajoro24} we dealt with the analysis  of the convolution operator $\lambda_1(\mu) $ for general locally compact groups. We now address the multiplication operator approach. That means that we look at the bounded  linear  operator $M_\phi\colon A(G)\to A(G)$ induced  by   a positive definite function $\phi$ on $G$,     given by  $ M_\phi u=\phi\cdot u$. This operator is said to be \textit{mean ergodic} when  $\frac1n (M_\phi + M_\phi^2+\cdots + M_\phi ^n )$ is convergent in the strong operator topology. If the strong operator topology is replaced by   the  operator norm topology, we say that $M_\phi$ is  \emph{uniformly mean ergodic}.

In this work, we particularly focus on  the   uniform problem.
 While our results parallel  the ones obtained  in \cite{gajoro24},  different approaches are often needed, especially due to the lack of  a viable transform for discrete groups (transforms were important in our treatment of compact groups in \cite{gajoro24})   and the absence of group structure in the quotient $G/H_\phi$, where $H_\phi$ denotes the closed subgroup $\phi^{-1}(\{1\})$. It is known that $M_\phi$ is mean ergodic precisely when $H_\phi$ is open. We show in this paper that uniform mean  ergodicity of $M_\phi$ is characterized through the spectral properties of    $M_\phi$, through the proximity of $\phi$  to a function supported on finitely many translates of $H_\phi$, and through its relation with operator quasi-compactness. Our approach is also used to characterize under which conditions the iterates $M_\phi^n$ of $M_\phi$ converge.

We next dualize the ergodicity properties of the random walk induced by a probability measure and model them  through $\phi$. This consists in analyzing the convergence to 0 of the means   $\frac1n (M_\phi u+ \cdots + M_\phi ^n u)$, for every $u\in A(G)$ with $u(e)=0$.
We prove that  $\phi$ is ergodic if and only if $\phi$ is  \textit{adapted}, i.e. if  $H_\phi =\{e\}$, precisely as in the commutative case and unlike the scenario  for general locally compact groups in   the convolution case.
This  equivalence can be deduced from  the work of Kaniuth, Lau and \"Ulger \cite{kanilauulger10}   and of Guex \cite{guex}, but  the scope of these proofs is somewhat blurred by nonessential hypotheses or  misguided connections. We provide  here a  short direct proof. Replacing in this context the strong operator topology by the uniform norm, we see that  $\phi $ is  uniformly ergodic if and only if $G$ is discrete and 1 is isolated in the spectrum of $\phi$. When it comes to powers of $\phi$ we prove that the operators $M_{\phi}^n$ converge to 0 uniformly on $A_0(G)$ if and only if $\phi$  is uniformly ergodic and $|\phi(x)|=1$ implies $x=1$.

\section{Preliminaries}
We need to establish the notation and basic facts concerning both ergodicity and  Fourier algebras.
\subsection{Preliminaries on mean ergodicity}
    For a Banach space $X$, we denote by $\ll (X)$ the space of linear and continuous operators from $X$ to itself and for a Hilbert space $\mathbb H$, the space of its unitary operators is denoted $\mathcal{U}(\mathbb H)$. The Ces\`aro means of $T\in \ll (X)$ are
        \[
            T_{[n]}=\frac1n \sum_{j=1} ^n T^j,
        \]
    where $T^j=T\circ \cdots \circ T$ denotes the $j$-th iterate of $T$. We say that  $T$ is \textit{ mean ergodic} when $(T_{[n]})_n$ converges in the strong operator topology to an operator $P\in \ll(X)$ (and then $P$ has to be the projection on the subspace of $X$ consisting on the vectors fixed by $T$), that is, when   $(T_{[n]}x)_n$ converges to $Px$   for each $x\in X$.  When $(T_{[n]})_n $ converges to $P$ in the operator norm, we say that $T$ is \textit{uniformly mean ergodic}. When $T$ is mean ergodic, the space decomposes as $X=\overline{(I-T)}(X) \oplus \ker (I-T)$. For further information on this topic, see the second chapter of \cite{kren85}.

     If $X$ is a Banach space and $T\in\mathcal{L}(X)$ we denote by $\sigma(T)$ its spectrum, i.e. $\sigma(T):=\{z\in\C:\ zI-T \text{ is not invertible}\}$. The resolvent mapping $R(\cdot,T):\C\setminus \sigma(T)\to \ll(X)$, $z\mapsto R(z,T):=(zI-T)^{-1}$ is holomorphic when $\ll(X)$ is endowed with its norm topology (cf \cite[Chapter VII]{conway}).

    We state here  a useful characterization of uniform mean ergodicity, due to the combined results of Dunford and Lin.

    \begin{theorem}
    \label{T:DL}
        Let $T \in \ll (X)$, where $X$ is a Banach space. If $(\|T^n\|)_n$ is bounded, then the following assertions are equivalent:
        \begin{enumerate}
            \item $T$ is uniformly mean ergodic,
            \item $(I-T)(X)$ is closed,
            \item either $1\not \in \sigma (T)$, or $1$ is a pole of order $1$ of the resolvent.
        \end{enumerate}

    \end{theorem}

    The equivalence of (1) and (3) was proved by Dunford in \cite{dunf43} and Lin proved the equivalence with (2) in \cite{lin74}.

   A property closely related to uniform mean ergodicity is quasi-compactness.  An operator $T\in \ll (X)$ is called \textit{quasi-compact} if there is a compact operator $K$ such that $\|T^n-K\|<1$ for some $n\geq 1$.  Yosida and Kakutani found that quasi-compactness provides sufficient conditions for uniform mean ergodicity and asymptotic convergence of the iterates of an operator. We record  this in the following theorem. Here, (and elsewhere in the paper) we will denote by $\sigma_p(T)$ the point-spectrum of $T$, the set of its eigenvalues, and by  $\T$  the set of complex numbers of modulus 1.

    \begin{theorem}[Theorem 4 and its Corollary in page 205 of \cite{yosikaku41}]
    \label{yosidakaku}
     Let $X$ be a complex Banach space and let $T\in \ll(X)$ be a quasi-compact operator. If $(\|T^n\|)_n$ is a bounded sequence, the following assertions hold:
     \begin{itemize}
         \item[(i)] $(T_{[n]})_n$ converges in norm to a finite rank projection $P$.
         \item[(ii)] $(T^n)_n$ converges in norm to a finite rank projection $P$ if, and only if,  $\sigma_p (T)\cap \T \subseteq \{1\}$.
         \item[(iii)] $(T^n)_n$ converges in norm to $0$ if, and only if, $\sigma_p (T)\cap \T =\emptyset$.
     \end{itemize}
     \end{theorem}
%
%

\subsection{Preliminaries on Fourier Algebras}
\label{PFA}

  The notation of \cite{kalau18} will be generally adopted here.

   Let $G$ be a locally compact group with identity $e$ and Haar measure $\mg$.

  \subsubsection{The group $C^\ast$ -algebra} The group  C$^*$-algebra of $G$, denoted by $C^*(G)$, is defined as the completion of $L^1(G)$ with the norm given by
    \[
        \|f\|_* = \sup _\pi \|\pi (f)\|,
    \]
    where the supremum is taken over the set of  all   unitary
    representations $\pi : G\longrightarrow \mathcal{U}(\h_\pi)$, and $\pi(f)\in \ll(\h_\pi)$, is defined by
      \[  \<{\pi (f)\xi,\eta} = \int _G \<{\pi(t)\xi,\eta} f(t) \dmg(t), \quad \xi,\eta \in \h_\pi.
    \]

\subsubsection{The Fourier and Fourier-Stieltjes algebras}
    The dual space of $C^*(G)$ can be identified with the Fourier-Stieltjes algebra $B(G)$
    consisting of functions on $G$
 of the form $\pi_\rho^\eta $,  where $\pi : G\longrightarrow \mathcal{U}(\h_\pi)$ is a continuous unitary representation of $ G$ and $\xi,\eta \in \h _\pi$, that are given by
    \[
        \pi_\xi^\eta(x) = \<{\pi (x) \xi, \eta}.
    \] The duality between  $C^*(G)$ and $B(G)$ is provided  by
    \[
        \<{f,\pi_\xi^\eta} =\int \pi_\xi^\eta(t)f(t) \dmg(t) = \<{\pi(f)\xi,\eta}, \quad f\in L^1(G),
    \] which also defines the norm of $B(G)$, by taking the supremum over all $f\in L^1(G)$, with $\|f\|_* \leq 1$.
 The functions  $\pi_\xi^\eta $ are known as  \emph{matrix coefficients} of the representation
  $\pi$ of $G$. When $\xi=\eta$ they are called \emph{diagonal} matrix coefficients. Diagonal matrix coefficients are positive definite functions, so that $B(G)$ is spanned by the set of positive definite functions on $G$.  We denote the set of continuous positive definite functions by $P(G)$ and by $P^1(G)$ those $\phi\in P(G)$ with $\phi (e)=1$. A useful summary of their properties can be found in Section 1.4 of \cite{kalau18}.
 We would like to display here one that will be especially useful for us, it  is proved in \cite[Theorem 3.7.7]{kalau18}:
 \begin{equation}\label{mult} \mbox{ if  $\phi(y)=1$, then } \phi(xy)=\phi(x) \mbox{ for every $x\in G$}.\end{equation}

Among the unitary representations of a locally compact group $G$, one is specially relevant for its capacity to carry the properties of the group. This is the left regular representation $\lambda_2\colon G\to  \mathcal{U}(L_2(G))$ defined by $\lambda_2(t)f(s)=(\delta_t\ast f)(s)=f(t^{-1}s)$, $s,t \in G$ and $f\in L_2(G)$. As happens with every unitary representation of $G$, $\lambda_2$ can be extended to a representation of $M(G)$. This extension is given by
\[ \lambda_2(\mu)f(s)=(\mu \ast f)(s)=\int_G f(t^{-1}s)d\mu(t)\quad s \in G.\]

    The Fourier algebra $A(G)$  can be described in several ways. We outline two of them here. For  a complete description we refer to \cite{eyma64} and Chapter 2 of \cite{kalau18}.
     If $C_{c}(G)$ denotes the space of continuous functions with compact support on $G$, one can define
    \[
        A(G)=\overline{\<{B(G)\cap C_c(G)}},
    \]
    where the closure is taken in the norm of $B(G)$.
    One can then prove that
    \[A(G)= \{f*\tilde g\: : \: f,g \in L^2(G)\},\]
        where, for a given function $g\colon G \to \C$,  $\tilde{g}\colon G\to \C$ is defined by $\tilde{g}(t)=\overline{g(t^{-1})}$.
    The
    elements of $A(G)$ can also be seen  as matrix coefficients of the left regular representation: if
      $u=f*\tilde g$, $f,g \in L_2(G)$,  then $u={(\lambda_2)}_{\overline{g}}^{\overline{f}}$. The Fourier algebra so defined becomes a closed ideal of the Fourier-Stieltjes algebra.

      The Lebesgue decomposition  of $B(G)$, first introduced in \cite[Remarque 3.20]{arsac76} and developed in \cite{kanilauschl03,miao99}, identifies  a closed linear subspace $B_s(G)$ of $B(G)$ such that every $\phi  \in B(G)$ can be expressed as $\phi=\phi_a+\phi_s$ with $\phi_a \in A(G)$ and $\phi_s\in B_s(G)$, and $\norm{\phi}=\norm{\phi_a}+\norm{\phi_s}$.
%

At a few points, we will be using two other norms on $B(G)$, the uniform  norm
$\norm{\phi}_\infty=\sup \{|\phi(x)|\colon x\in G\}$
and the \emph{multiplier norm}. The multiplier norm is defined by \[\norm{\phi}_{MA(G)}=\norm{M_\phi}.\]
With this norm, the algebra $B(G)$ is a subspace of the \emph{multiplier algebra}  $M(A(G))$ of $A(G)$, made of those operators $S\in \ll(A(G))$ with $S(uv)=S(u)v$ for every $u,v\in A(G)$.
It is well-known that such an $S$ is always a multiplication operator by  some bounded continuous  function on $G$ (see, for instance, Theorem 1.2.2 of \cite{lars71}).

The norms $\norm{\cdot}_{M(A(G))}$ and $\norm{\cdot}_{B(G)}$  are equivalent (and even coincide)  if and only if $G$ is amenable. Amenability of $A(G)$ is also a necessary and sufficient condition for $M(A(G))=B(G)$,  see \cite[Chapter 5]{kalau18} for further information on $MA(G)$.
It will be useful to  record here that, for any locally compact group and $\phi\in B(G)$,
\begin{equation}\label{comp.norms}
\norm{\phi}_\infty\leq\norm{\phi}_{MA(G)}\leq \norm{\phi}_{B(G)}.
\end{equation}
\subsubsection{The dual spaces $A(G)^\ast$ and $B(G)^\ast$}The Banach space dual of $A(G)$ can be identified with the   von Neumann algebra of $G$, denoted by $\VN (G)$. This algebra is defined as the closure, in the weak operator topology,  of $\lambda_2(L^1(G))$
      and its identification with  the dual space of $A(G)$ is realized through  the duality
    \[
        \<{\lambda_2(f), (f_1*\tilde f_2)^\vee} = \<{\lambda_2(f) f_1, f_2}_{L_2(G)},\quad f\in L_1(G),\, (f_1*\tilde f_2)^\vee\in A(G).
    \]

    When $G$ is commutative and $\widehat{G}$ is the group of characters of $G$, the Fourier-Stieltjes transform establishes a linear isometry between     $B(G)$  and $M(\widehat{G})$ and between $A(G)$ and $L^1(\widehat{G})$.
    The same identifications makes the algebra $C^\ast(G)$  isometrically isomorphic to $C_0(\widehat{G})$ and $VN(G)$ to $L_\infty(\widehat{G})$.

   The dual space of  $B(G)^\ast $ can be  canonically identified with the universal enveloping von Neumann algebra of $C^\ast(G)$. We give some details here on that construction, see, e.g., \cite[Remark 2.1.6]{kalau18}, or \cite[Section III.2]{take02}, for full details and proofs.
    The  GNS construction associates to  each $\phi \in P^1(G)$, a specific representation, $\pi_\phi \colon G\to \mathcal{U}(\h_{\phi})$ and a vector $v_\phi \in \h_\phi$ in such a way that $\phi=(\pi_{\phi})_{v_\phi}^{v_\phi}.$
    The \emph{universal representation} is  the representation of $G$, $\omega\colon G \to \mathcal{U}\bigl(\h_{\mathrm{uni}}\bigr)$  by unitary operators on the
    Hilbert space $\h_{\mathrm{uni}}:=\bigoplus_{\phi \in P^1(G)}\h_\phi$
 given by \[ \omega(g)\biggl( \sum_{\phi\in P^1(G)} w_\phi\biggr)=\sum_{\phi \in P^1(G)} \pi_\phi(g)w_\phi, \mbox{ where }w_\phi\in \h_\phi.\]
     The double commutant of $\omega(G)$ in $\ll\bigl(\h_{\mathrm{uni}}\bigr)$ is then a von Neumann algebra, the \emph{universal enveloping von Neumann algebra } of $C^\ast(G)$, denoted as $W^{\ast}(G)$. From the identification of $B(G)$ with $(C^\ast(G))^*$, one can obtain a natural identification
      \[ \<{\omega(g),\psi}=\psi(g),\quad \mbox{ for any $\psi\in B(G)$} .\]
      The elements of $W^\ast(G)$ can be ultraweakly approximated by linear combinations of operators in $\omega(G)$. It follows that every unitary representation of $\pi \colon G\to \mathcal{U}(\h)$,  can be extended  to a representation $\pi''\colon W^\ast(G)\to \ll(\h)$. The     duality between $W^\ast(G)$ and $B(G)$ is then  given by the relation
      \begin{equation} \label{defb*}\<{L,\phi}=\<{\pi''(L) \xi,\eta}\quad \mbox{ for each } L\in W^*(G) \mbox{ and }\phi=\pi_\xi^\eta \in B(G).\end{equation}

       The preceding construction    furnishes   $W^\ast(G)$ with a naturally  defined   multiplication, the multiplication of operators.
There is a different way to address the  multiplication of elements in $W^\ast(G)$,  one that can be used in the second dual of any  Banach algebra. It  was introduced, in a more general form, in \cite{aren51} and has since been  known as the Arens-construction, or the Arens multiplication. It is well known, see \cite[Theorem 7.1]{civinyood61}, that for $\mA=C^\ast(G)$, this multiplication coincides with  the operator multiplication that  $W^\ast(G)$ acquires as the enveloping algebra of $C^\ast(G)$.

       The  Arens multiplication on the bidual  $\mA^\bd$  of a Banach algebra  $\mA$ is defined after introducing two module actions on $\mA^\ast$,  one by elements of  $\mA$ and the other by elements of  $\mA^\bd$. These actions, along with  the Arens multiplication, are presented  below, for  $a,b\in \mA$, $\phi\in \mA^\ast$ and $T,T_1,T_2\in \mA^\bd$.
       \label{module}
       \begin{align*}
         \phi \cdot a &\in \mA^\ast  \mbox{ is defined by } \<{\phi \cdot a,b}=\<{\phi,ab},\\
         T\cdot \phi &\in \mA^\ast  \mbox{ is defined by } \<{T\cdot\phi ,a}=\<{T,\phi\cdot a},\\
         T_1\cdot T_2 &\in \mA^\bd  \mbox{ is defined by } \<{T_1\cdot T_2 ,\phi}=\<{T_1,T_2\cdot \phi}.
       \end{align*}
       It is clear from the above definitions that
       \begin{align*}
         \norm{\phi\cdot a}&\leq \norm{\phi }\cdot \norm{a},\\
         \norm{T\cdot \phi}&\leq \norm{T }\norm{\phi },\\
         \norm{T_1\cdot T_2}&\leq \norm{T_1 }\norm{T_2}.
       \end{align*}

%

  %

%
%
%
%
%
If $\mA$ is a  Banach algebra,  $\Delta (\mathcal{A})$ will always denote its spectrum, i.e., the set of multiplicative bounded functionals of $\mathcal{A}$.




\subsection{Preliminaries on ergodicity in Fourier algebras}
We outline here the  concepts that will let us dualize the ergodic theory of random walks on groups.

Every function $\phi\in B(G)$ defines a bounded linear operator $M_\phi\colon A(G)\to A(G)$ given by $M_\phi u(s)=u(s)\phi(s)$, $u\in A(G),\;s\in G$.
If  we  define the augmentation ideal in $A(G)$,
    \[
        A_0(G) = \{u\in A(G)\: : \: u(e)=0\},
    \] then $A_0(G)$ is stable under the action of $M_\phi$. We denote by $M_\phi^0\colon A_0(G)\to A_0(G)$ the restriction of $M_\phi$ to $A_0(G)$.

We will also write, for $\phi \in B(G)$, $\phi_{[n]}=\frac{1}{n}(\phi+\cdots+\phi^n)$ and will use interchangeably the expressions $(M_\phi)_{[n]}$ and $M_{\phi_{[n]}}$, as well as the expressions
$(M_\phi)^n$ and $M_{\phi^n}$.


\begin{definition}Let $G$ be a locally compact group.
We say that $\phi \in P^1(G)$ is:
    \begin{itemize}
        \item ergodic if $\lim_n M_{\phi_{[n]}} ^0 u =0$, for every $u\in A_0(G)$,
        \item uniformly ergodic if $\lim_n\|M_{\phi_{[n]}} ^0\|=0$,
        \item completely mixing if  $\lim_n M_{\phi^n} ^0 u =0$, for every $u\in A_0(G)$,
        \item uniformly completely mixing if $\lim_n\|M_{\phi^n} ^0\| =0$.
    \end{itemize}
\end{definition}
A longstanding objective in the study of the random walk induced by a probability  measure $\mu$  has been to classify its   ergodic behaviour through algebraic properties of the support of $\mu$. Among these  properties two stand out: adaptedness, the support of $\mu$ is not contained in any proper closed subgroup of $G$, and strict aperiodicity, the support of $\mu$ is not contained in any translate of a proper closed normal subgroup of $G$.   When $G$ is abelian, these properties can be characterized by properties of the Fourier-Stieltjes transform $\widehat{\mu}$. We take these characterizations as definitions on positive definite functions:
\begin{definition}     Let $G$ be a locally compact group and let $\phi\in P^1(G)$. We define the sets $H_\phi= \{x\in G\: : \: \phi(x) =1\}$ and $E_\phi =\{x\in G\: : \: |\phi(x)|=1\}$. We say $\phi$ is:
\begin{itemize}
    \item \emph{adapted,} if $H_\phi = \{e\}$,
    \item \emph{strictly aperiodic,} if $E_\phi =\{e\}$.
\end{itemize}
\end{definition}
These properties characterize the convergence of the means and powers of $M_\phi$ in the strong operator topology, see \cite[Theorems 2.2 and 2.8]{musta}. By well-known results, e.g. \cite[Theorem 3.7]{kalau18}, this is also  equivalent to the convergence of the means and powers of $\phi$ in the compact-open topology.
\begin{proposition}
\label{P:meOpen}
    Let $G$ be a locally compact group and let $\phi\in P^1(G)$. The following are equivalent.
    \begin{itemize}
   \item[(i)]  $M_\phi$ is mean ergodic.
   \item[(ii)] $H_\phi$ is an open set.
   \item[(iii)] $(\phi_{[n]})$ is convergent to $\Cf_{H_\phi}$ in the compact open topology.
    \end{itemize}
\end{proposition}
\begin{proposition} \label{iteratessot} Let $G$ be a locally compact group and let $\phi\in P^1(G)$. The following assertions are equivalent.
\begin{itemize}
 \item[(i)] $(M_{\phi^n})_n$ is convergent in $\mathcal{L}(A(G))$ endowed with the strong operator topology.
 \item[(ii)] $H_\phi=E_\phi$ is an open set
 \item[(iii)]  $(\phi^n)_n$ is convergent to $\Cf_{H_\phi}$ in the compact open topology.
 \
  \end{itemize}
\end{proposition}

 \section{Uniform mean ergodicity of $M_\phi$.}
 In our paper \cite{gajoro24}, the uniform mean  ergodicity of the convolution operator $\lambda_1(\mu)$ was characterized in terms of the position of 1 in its spectrum, the nonsingularity of convolution powers of $\mu$ and quasicompactness of $\lambda_1(\mu)$. Here, we return  to these  perspectives in the context of multiplication operators.
\subsection{Uniform mean ergodicity and  the spectrum}
By the Dunford-Lin Theorem, Theorem \ref{T:DL},
 $1$ is isolated in $\sigma(T)$ whenever $T$ is a uniformly mean ergodic operator.
 The converse may fail for general operators (the operator $T=I-V$ with $V$ being the Volterra operator is such an example, see
 \cite[Remark 5.26 (a)]{gajoro24}) but it does hold in the case of convolution operators, see \cite[Theorem 5.5]{gajoro24} and, as we prove in this section,  in the case of  the  multiplication operators $M_\phi$ discussed here. 

  Our main tools will be the two following general results.
\begin{theorem}[\cite{gro04},Theorem 1; \cite{bofrejo07},Corollary 10, Remark 11]
\label{P:extension}
    Let $X$ be a Banach space, $H\subseteq X^*$ a separating subspace, $\Omega\subseteq \C$ a domain and $a\in \Omega$. If $f:\Omega\setminus\{a\} \longrightarrow X$ is a holomorphic function such that $x^* \circ f$ admits a holomorphic extension to $\Omega$ for each $x^*\in H$, then $f$ admits a holomorphic extension to $\Omega$.
\end{theorem}
We apply this Theorem to subalgebras $\mA\subseteq C_b(Y)$ of the space of bounded continuous functions of a topological space $Y$. Under the conditions of the proposition below, we use the family $\{\delta_{y,f}\: : \: y\in Y, f\in \mA\}\subseteq \ll (\mA)^*$, given by $\<{\delta_{y,f},T}=Tf(y)$, as a separating subspace of $\ll(\mA)$ (i.e., for different $T_1,T_2\in\ll(\mA)$ there are $y,f$ such that $T_1f(y)\neq T_2f(y)$).

\begin{proposition}
 \label{cb}
 Let $Y$ be a Hausdorff, normal locally compact space and let $\mA$ be a  subalgebra of $ C_b(Y)$  equipped with a Banach algebra norm $\norm{\cdot}$ satisfying $\norm{\cdot}\geq \norm{\cdot}_\infty$ and
 \begin{equation}
   \label{cond}   \mbox{ for each $y\in Y$ there is $f\in\mA$ with $f(y)\neq 0$}.
 \end{equation}
 If  $\psi\in C_b(G)$  is such that $M_\psi:\mA\to\mA$, $f\mapsto \psi f$ is a bounded  linear operator with $\|M_\psi\|\leq 1$,  then $M_\psi$ is uniformly mean ergodic if and only if $1$ is not an accumulation point in $\sigma(M_\psi)$.
 \end{proposition}
 \begin{proof}
We only have to show that $M_\psi:\mathcal{A}\to \mathcal{A}$ is uniformly mean ergodic when $1$ is not an accumulation point in $\sigma(M_\psi)$.
In this proof we will need to work  with two resolvent operators: for a given $z\in \C$,
we will denote by $R^{C_b(Y)}(z,M_\psi)$ the resolvent operator with $M_\psi$ seen as an operator  in $\ll(C_b(Y))$, and  by $R ^{\mA}(z,M_\psi) $    the resolvent operator  with values in $\ll(\mA)$. We use analogous notation for the spectra. We observe that  \eqref{cond} implies that $\overline{\psi\left(Y\right)}\subseteq \sigma^\mA(M_\psi)$. It is also well known that $\sigma^{C_b(Y)}(M_\psi)=\overline{\psi(Y)} $. As $\sigma^{C_b(Y)}(M_\psi)=\overline{\psi(Y)}\subseteq \sigma^\mA(M_\psi)$,  for each $f\in \mA\subseteq C_b(Y)$ and $y\in Y$ we get that
\begin{align*}
  R^{C_b(Y)}(z,M_\psi)f(y)&=\frac{f(y)}{z-\psi(y)}\in C_b(Y), \mbox{ for each $z\in \C\setminus \overline{\psi(Y)}$},\\
  R ^{\mA}(z,M_\psi)f(y)&=\frac{f(y)}{z-\psi(y)}\in \mA, \mbox{ for each $z \in \C\setminus \sigma^\mA(M_\psi)$}.
\end{align*}
 By hypothesis, there is $R>0$ such that $B(1,R)\cap \sigma^\mA(M_\psi)=\{1\}$. This yields that $|1-\psi(y)|\geq R >0$, for each $y\in Y$ with $\psi(y)\neq 1$. Furthermore, the set $\psi^{-1}(1)=\psi^{-1}(B(1,R))$ is then  open in $Y$ and  the multiplication operator $M_\psi:C_b(Y)\to C_b(Y)$ is uniformly mean ergodic by \cite[Theorem 2.7]{bjr18}.


 Then,  by using (1) implies (3) in  Theorem~\ref{T:DL},  we get that $(\cdot-1)R^{C_b(Y)}(\cdot,M_\psi): B(1,R)\setminus\{1\}\to \mathcal{L}(C_b(Y))$  admits holomorphic extension in $z=1$.

 We have that, for each $z\in B(1,R)$,
\[\<{\delta_{y,f}, (z-1)R ^{\mA}(z,M_\psi)}=\<{\delta_{y,f}, (z-1) R ^{C_b(Y)}(z,M_\psi)},\] for every $y\in Y$ and every $f\in \mA$.

 Having seen that $\delta_{y,f}\circ (\cdot-1) R^{\mA}(\cdot,M_\psi)$  is holomorhic in $B(1,R)$ and, taking into account that
 $\{\delta_{y,f}:\ y\in Y, f\in\mathcal{A}\}\subseteq  \mathcal{L}(\mathcal{A})^*$ is separating,  we deduce from Theorem \ref{P:extension} that 1 is a pole of order 1 of $R^{\mA}(\cdot,M_\Psi)$. We conclude by applying (3) implies (1) in Theorem~\ref{T:DL}.
  \end{proof}

We can now characterize uniform mean ergodicity of the operators $M_\phi$. The  same approach can be used to  characterize the convergence of the means in the norm of $B(G)$.

\begin{theorem}
\label{P:UMEequiv1isol}
   Let $G$ be a locally compact group and let $\phi\in P^1(G)$.
   \begin{itemize}
   \item[(i)] $M_\phi \in \ll(A(G))$ is uniformly mean ergodic if and only if $1$ is isolated in $\sigma (M_\phi)$.
   \item[(ii)] $(M_{\phi^n})_n$ is convergent in the norm topology if and only if $\sigma(M_\phi)\cap\T=\{1\}$ and $1$ is isolated in $\sigma (M_\phi)$.
   \item[(iii)] $(\phi_{[n]})$ is convergent to $1_{H_\phi}$ in $B(G)$ if and only if $1$ is isolated in $\sigma (\phi)$.
   \item[(iv)] $(\phi^{n})$ is convergent to $1_{H_\phi}$ in $B(G)$ if  and only if $\sigma(M_\phi)\cap\T=\{1\}$ and $1$ is isolated in $\sigma (\phi)$
   \end{itemize}
\end{theorem}
\begin{proof}
All the statements follow from  Proposition \ref{cb} with $A(G)\subseteq B(G)\subset C_b(G)$. For Statement (i), we apply  Proposition \ref{cb} applied to $M_\phi:A(G)\to A(G)$ and for Statement (iii), we apply  Proposition \ref{cb} applied to $\widetilde{M}_\phi:B(G)\to B(G), f\mapsto \phi \cdot f$. Note that    $\| \widetilde{M}_\phi\|=\|\phi\|$ (since $1\in B(G)$) and  $\sigma(\widetilde{M}_\phi)=\sigma(\phi)$. We obtain statements (ii) and (iv) after combining   statements (i) and (iii), respectively, with  \cite[Corollaire 3]{mbezem93}.
\end{proof}

\begin{remark}
The hypothesis $\phi\in P_1(G)$  in Theorem \ref{P:UMEequiv1isol} can be relaxed. For instance, in (i) and (ii) we only need $\phi\in MA(G)$ and $(\norm{M_\phi^n})_n$ power bounded, and in (iii) and (iv) $\phi\in B(G)$ and $(\|\phi^n\|)_n$ bounded.
\end{remark}

\subsection{Uniform mean ergodicity and spread-out functions}
If $G$ is amenable and    $\phi$ is adapted (or even if $H_\phi$ is compact), uniform mean ergodicity of $M_\phi$ implies that the means of $\phi$  converge to an element of $A(G)$ (namely, the characteristic function of $H_\phi$).  Since the singular part in the Lebesgue decomposition  $B(G)=A(G)\bigoplus B_s(G)$ is a closed vector subspace, it follows immediately that some power of $\phi$ is not going to be singular. This  is the original argument for convolution operators.
 In this latter case,  the converse is proved using the connection between adapted spread-out measures (measures with some non-singular convolution power) and quasi-compact operators (see next subsection for this theme). If the measure is not adapted, one may consider the same measure restricted to its support group $H_\mu$ (the smallest closed subgroup of $G$ that contains the support of $\mu$) and  show that $\lambda_1(\mu)$ is uniformly mean ergodic if and only if some power of $\mu$ is not singular with respect to the Haar measure of $H_\mu$ (see Remark 4.8 of \cite{gajoro24}).

This approach faces some difficulties in the case of the multiplication operators $M_\phi$ on general locally compact groups. The reason for this  is twofold.

 Firstly, the norms of $M_\phi$ and $\phi$ do not necessarily coincide in nonamenable groups, making the Lebesgue decomposition less useful in that case.  We will instead lean on the characters of $MA(G)$ to deduce that, given  an adapted $\phi\in B(G)$, for the multiplier $M_\phi$ to  be uniformly mean ergodic, it is necessary and sufficient that, for some $n\in \N$, $\phi^n$ is \emph{not far, } in the multiplier norm, from $A(G)$.

 Secondly, the reduction to the adapted case is not as straightforward as for  convolution operators.
The natural path here  would  be to work with the quotient $G/H_\phi$, but this  may fail to be a  group. One way of remedying this could be to replace $A(G)$ and $B(G)$ by $A_{\lambda_{G/H}}$ and $B_{\lambda_{G/H}}$ as defined by Arsac, \cite{arsac76}, with $\lambda_{G/H}$ denoting the quasi-regular representation of $G/H$. But this option brings some pathologies, such as $A_{\lambda_{G/H}}$ not being an ideal (nor a subalgebra) in  $B_{\lambda_{G/H}}$.
The \emph{right} approach turns out to be recurring to the algebras $A(G/H)$ and $B(G/H)$ introduced in \cite{forr98}.
\begin{definition}[Forrest \cite{forr98}]Let $H$ be an open subgroup of the locally compact group $G$ and let $p\colon G\to G/H$  denote the quotient map.
\begin{itemize}
  \item The \emph{Fourier-Stieltjes algebra} of $G/H$ is defined as
  \[    B(G/H)= \{ \psi\in B(G) \: : \: \psi(xh)=\psi(x), \, x\in G, h\in H \}.\]
\item The \emph{Fourier algebra} of $G/H$ is defined as
\[
    A(G/H) = \overline{\langle  \psi \in B(G/H) \: : \: p(\supp \psi) \text{ is finite}   \rangle}.
\]

\end{itemize}
\end{definition}
Note that
\[
    A(G/H)=\overline{\langle\Cf_{xH} , x\in G\rangle},
\]
and that this definition of $A(G/H)$ brings back $A(G)$ when $H=\{e\}$.

We can  now  give sense to the notion of spread-out  functions, in the nonadapted case.
\begin{definition}
  Let $G$ be a locally compact group and let $\phi\in P^1(G)$. Consider the quotient map $p\colon G\to G/H_\phi$ and let $\tilde{\phi}:G/H\to \C$ be defined so that  $\phi= \tilde{\phi}\circ p$.  We say that $\tilde{\phi}$ is \emph{spread-out} if there exist $k\in \N$ and $v\in A(G/H_\phi)$ such that $\|M_\phi^k-M_v\|< 1$. We say that $\phi$ is spread-out when $v$ can be taken in $A(G)$.
\end{definition}
It is clear from the preceding definition that, when $G$ is amenable, $\phi$ is spread-out if and only if
$\phi^k\notin B_s(G)$ for some $k\in \N$.

The following fact, a straightforward adaptation of \cite[Theorem 2.3.8]{kalau18}, will be useful in the proof of Theorem \ref{T:umeCaractHso} below.
\begin{lemma}
\label{L:spectroAGH}
   Let $G$ be a locally compact group and let $H$ be an open subgroup of $G$. If $0\neq \chi\in   \Delta(A(G/H))$, then there is $x\in G$ with $\chi(u)=u(x)$ for every $u\in A(G/H)$.
\end{lemma}
\begin{proof} Let $\chi\in \Delta(A(G/H))$ and assume $\chi = T_x$  for no $x\in G$, where $T_x\in \Delta(A(G/H)$ is the point evaluation $T_xu=u(x)$ for $u\in A(G/H)$. Then (see the proof of \cite[Theorem 2.1.8]{kan08}), $\ker \chi \neq \ker T_x$ for all $x\in G$. So, for each $x\in G$ there exists $f_x\in A(G/H)$ such that $\chi(f_x)=1$ and $f_x(x)=0$.

By density of the functions with finite support in $G/H$, we may assume that there is $f_0= \sum _{j=1} ^N \alpha_j \Cf _{x_j H}$  with $\chi(f_0)=1$.

    Now define $f=f_0\cdot f_{x_1} \cdots f_{x_N}\in A(G/H)$. The definition of $f$, forces $f= 0$, since $f_{x_j}(x_jh)=0$, for all $h\in H$, and $f_0(x)=0$ whenever $x\notin \bigcup_{j=1}^N x_jH$. However, $\chi (f) = \chi(f_0) \cdot \chi (f_{x_1}) \cdots \chi (f_{x_N}) =1$, a contradiction.
\end{proof}

We now study the stability under multiplication  of those elements of $W^\ast(G)$ that are  bounded for the multiplier norm. We first give them a name.
\begin{definition}
  Let $G$ be a locally compact group. We define
  \[W^\ast_M(G)=\bigl\{ T \in W^\ast(G) \colon \left| \<{T,\psi}\right|\leq C \norm{M_\psi} \mbox{ for every } \psi\in B(G)\bigr\}.\]
\end{definition}

If $T\in W^\ast_M(G)$, $\norm{T}_{W^\ast_M(G)}$, is defined as
\[\sup\bigl\{ \left| \<{T,\psi}\right|\colon \psi \in B(G)  \mbox{ and } \norm{M_\psi}\leq 1\bigr\}.\]
\begin{lemma}\label{prodm}
$W^\ast_M(G)$ is a $\ast$-subalgebra of $W^\ast(G)$.
\end{lemma}
\begin{proof}
The set $W^\ast_M(G)$ is clearly a vector subspace of $W^\ast(G)$.

  We now  show that $T_1,T_2\in W_M^*(G)$ implies that  $T_1\cdot T_2\in W^\ast_M(G)$.

  If $\psi \in B(G)$, then
  \begin{align*}
    \bigl|\<{T_1\cdot T_2,\psi}\bigr|&= \bigl| \<{T_1, T_2\cdot \psi}\bigr|\\
    &\leq \norm{T_1}_{W^\ast_M(G)}\cdot  \norm{M_{T_2\cdot \psi}}\\
    &= \sup_{\overset{u\in A(G)}{\norm{u}\leq 1}} \norm{T_1}_{W^\ast_M(G)}\cdot \norm{T_2\cdot \psi \cdot u}_{B(G)}\\
     &\leq \sup_{\overset{u\in A(G)}{\norm{u}\leq 1}} \norm{T_1}_{W^\ast_M(G)}\cdot \norm{T_2}_{W^\ast(G)} \cdot \norm{\psi \cdot u}_{B(G)}\\
    &\leq  \norm{T_1}_{W^\ast_M(G)}\cdot \norm{T_2}_{W^\ast(G)} \cdot \norm{M_\psi}.
  \end{align*}
  And $T_1\cdot T_2\in W^\ast_M(G)$.

That $T^\ast\in W_M^\ast(G)$ whenever $T\in W_M^\ast(G)$ is a consequence of    $\<{T^\ast,\phi}=\<{\pi''(T^\ast)\xi, \xi}=\overline{\<{T,\phi}}$ for every $\phi=\pi_\xi^\xi\in \mathcal{P}(G)$, and hence  for every $\phi\in B (G)$.
\end{proof}
For our next proof, we need a description of the quasi-regular representation $\lambda_{G/H}$. If $H$ is an open subgroup of a locally compact group, $\lambda_{G/H}\colon G \to \mathcal{U}\bigl(\ell_2(G/H)\bigr)$ is given by
\[ \lambda_{G/H} (x) f (tH)=f(x^{-1}tH), \quad f\in \ell_2(G/H),\; x\in G,\; tH\in G/H.\]
If $\delta_{xH}$, $x\in G$,  denotes the element of $\ell_2(G/H)$ that takes the value 1 on $xH$ and 0 elsewhere, then

\begin{equation}
\label{eq}
\bigl(\lambda_{G/H}\bigr)_{\delta_{xH}}^{\delta_{yH}}=\cf_{yHx^{-1}}.
\end{equation}

We have therefore that  coefficients of $\lambda_{G/H}$ need not belong to $B(G/H)$. They do, however, belong  when $xH=H$.
\begin{lemma}\label{annih}
 Let $G$ be a locally compact group with an open subgroup $H$ and let $T_1,\, T_2\in W^\ast(G)$.  If $\<{T_2,u}=0$ for every $u\in A(G/H)$, then $\<{T_1\cdot T_2,\cf_{H}}=0$.
\end{lemma}
\begin{proof}
  We first observe that, given $f=\sum_{i=1}^N \alpha_i \delta_{x_iH}\in \ell_2(G/H)$ with finite  support (in $G/H$), taking into account that $    \sum_{i=1}^N \overline{\alpha_i} \cf_{x_iH}\in A(G/H)$, we use \eqref{eq} and the identification in \eqref{defb*}  to get
  \begin{align*}
     \<{\lambda_{G/H}''(T_2)\delta_{H},f}&=
     \<{T_2,\left(\lambda_{G/H}\right)_{\delta_H}^{f}}\\&=
     \<{T_2, \sum_{i=1}^N \overline{\alpha_i } \cf_{x_iH}}=0.
   \end{align*}
  Since functions with finite support are dense in $\ell_2(G/H)$, we deduce that $\lambda_{G/H}''(T_2)\delta_{H}=0$.
  Hence,
   \begin{align*}
    \<{T_1\cdot T_2,\cf_H} & =\<{\lambda_{G/H}''(T_1\cdot T_2)\delta_{H},\delta_{H}}\\
    &=\<{\lambda_{G/H}''(T_1)\bigl(\lambda_{G/H}''(T_2)\delta_{H}),\delta_{H}}=0.
    \end{align*}
    %
\end{proof}
\begin{theorem}
\label{T:umeCaractHso}
    Let $G$ be a locally compact group and let $\phi \in P^1(G)$. The following assertions are equivalent:
    \begin{enumerate}
        \item The operator $M_\phi$ is uniformly mean ergodic.
         \item The subgroup $H_\phi$ is open and the  function $\widetilde\phi$ is spread-out .
    \end{enumerate}
\end{theorem}
\begin{proof}
 Assume  that  $M_\phi$ is uniformly mean ergodic. The subgroup $H_\phi$ must be open by Proposition \ref{P:meOpen}. Suppose, towards a contradiction, that  $\tilde\phi$ is not spread-out.

    We consider then the Banach algebra  $M(A(G/H_\phi))$ consisting of  the multipliers $M_f\in M(A(G))$ with $f$ constant on the left cosets of $H_\phi$. The algebra $A(G/H_\phi)$ is then a closed ideal of $M(A(G/H_\phi))$. Let $M/A$ and $B/A$ denote, respectively,  the quotients of $M(A(G/H_\phi))$ and $B(G/H_\phi)$ by  $A(G/H_\phi)$  .  A consequence of  $\tilde\phi$ not being spread-out is that, for every $j\in \N$
    \[ \norm{\phi^j}_{\frac{M}{A}}=1.\]
    The  spectral radius of $\phi$ in the quotient algebra $M/A$ must therefore be 1 and, by  Gelfand duality, there must be a character $T\in \Delta(M/A)$ with  $\left|\<{T,\phi^j}\right|=1$, for every $j\in \N$. This character can be extended by the Hahn-Banach theorem to a functional in the unit ball of  $MA(G)^\ast$ which, when restricted to $B(G)$,  yields  a functional on $B(G)$. We keep the name $T$ for this functional.
    It is clear from the construction of $T$ that $T\in W_M^\ast(G)$. It is also clear  that
    \begin{align*}
    & \left|\<{T,\phi^j}\right|=1,\mbox{ for every } j\in \N \mbox{ and }\\
     & \<{T,u}=0, \mbox{ for every } u\in A(G/H_\phi).
    \end{align*}
   Choose now, for each $j\in \N$, a unitary representation $\pi_j\colon G\to \mathcal{U}(\h_j)$  and $\xi_j \in \h_j$, $\norm{\xi_j}=1$,  so that $\phi^j=(\pi_j)_{\xi_j}^{\xi_j}$.  We see then that, for each $j\in \N$,
  \[\left|\<{\pi_j''(T)\xi_j,\xi_j}\right| =\left|\<{T,\phi^j}\right|=1.\]
 Since   $1=\norm{\xi_j}\geq \norm{\pi''(T)\xi_j}$, we deduce that $\pi''(T)\xi_j=\<{T,\phi^j} \xi_j.$

 We next consider the operator $T^\ast T$.
 \begin{align*}
   \<{T^\ast T ,\phi^j}&=\<{\pi''(T^\ast T)\xi_j,\xi_j}\\
  &=\<{\pi''(T)\xi_j,\pi''(T)\xi_j}\\&=|\<{T,\phi^j}|^2=1.
  \end{align*}
Since,  by Lemma \ref{prodm}, $T^\ast T\in W_M^\ast(G)$ and, by Lemma \ref{annih}, $\<{T^\ast T,\cf_{H_\phi}}=0$, we obtain the following contradiction with the uniform mean ergodicity of $M_\phi$.
\begin{align*}
  1&=\<{T^\ast T,\phi_{[n]}}\\
  &=\<{T^\ast T,\phi_{[n]}-\cf_{H_\phi}}\\
  &\leq \norm{T^\ast T}_{W^\ast_M(G)} \cdot \norm{M_{\phi_{[n]}}-\cf_{H_\phi}}.
\end{align*}
We have thus proved that Assertion (1) implies Assertion (2).

    Assume next  that $\widetilde\phi$ is spread-out.
    This means that there exist $\phi_a\in A(G/H_\phi)$ and $\phi_s\in B(G/H_\phi)$ with $\|M_{\phi_s}\|<1$ such that $\phi ^k = \phi _a + \phi_s$, for some $k\in \N$.

  Suppose that  $(z_n)_n\subseteq  \sigma(M_\phi)$ is a sequence  with $\lim_n z_n=1 $. There will be then
   $\chi_n \in \Delta M(A(G))$ such that $\chi_n (\phi)=z_n $.
  Since, for every $n\in \N$,   $\restr{\chi_n}{B(G)}\in W^\ast_M(G)$
    \[
        |\chi_n (\phi_a)| = |z_n^k - \chi_n(\phi_s)| \geq |z_n|^k - \|\phi_s\|_{M(A(G))}.
    \]
  We can  therefore find $\alpha>0$ and  $n_0$ big enough so that $|\chi_n (\phi_a)| >\alpha >0$, for all $n\geq n_0$.

    Now, by Lemma~\ref{L:spectroAGH}, for each $n\in\N$ there is  $x_n\in G$ such that $\chi_n(\phi_a)=\phi_a(x_n)$. As  $|\phi_a (x_n)|>\alpha$, for all $n\geq n_0$, and
    functions in $A(G/H)$ must vanish at infinity in $G/H$  ($A(G/H)$ is spanned by functions of finite support in a norm stronger that the uniform norm), we see that
    $\{x_n H_\phi\: : \: n\in\N\}$ is a finite set. The sets
    $\{\phi_a(x_n) \colon n\in \N\}$ and $\{\phi_s(x_n) \colon n\in \N\}$ will have to be finite as well, for both $\phi_a$ and $\phi_s$ are constant on cosets of $H_\phi$.

    Since $A(G/H_\phi)$ is an ideal of $B(G/H_\phi)$, $\phi_a\phi_s\in A(G/H_\phi)$ and we have that
    \[
     \phi_a(x_n)\chi_n(\phi_s)=   \chi_n (\phi_a)\chi_n (\phi_s)= \chi_n (\phi_a\phi_s) = \phi_a(x_n)\phi_s(x_n),
    \]
    so that $\chi_n (\phi_s)=\phi_s(x_n)$, for any $n\geq n_0$.

    Recalling that  $z_n ^k = \chi_n (\phi^k) = \phi_a (x_n) + \phi_s (x_n)$, we find that  the set  $\{ z_n\colon n\in \N\}$ has to be  finite. As $(z_n)_n$ was an arbitrary sequence in $\sigma(M_\phi)$ that  approximates $1$, we conclude that  $1$ is isolated in $\sigma (M_\phi)$.
  Theorem \ref{P:UMEequiv1isol} shows that $M_\phi$ is uniformly mean ergodic.


\end{proof}

\begin{corollary}\label{C:umeCaractHso}
 Let $G$ be a locally compact group and let $\phi \in P^1(G)$. Then the sequence of means $(\phi_{[n]})_n$ converges to $\cf_{H_\phi}$ if and only if  $H_\phi$ is open and, for some $k\in \N$, there is $u\in A(G/H_\phi)$ such that $\norm{\phi^k-u}<1$.
\end{corollary}
\begin{proof}
The proof of Theorem \ref{T:umeCaractHso} can be used here replacing $M(A(G))$ by $B(G)$, and using the right version of Theorem \ref{P:UMEequiv1isol} throughout.
\end{proof}
 \begin{example}
 \label{ex}
 The operator $M_\phi$ can be uniformly mean ergodic, even if the sequence $(\phi_{[n]})_n$ is not convergent to  $\cf_{H_\phi}$.
\end{example}
\begin{proof}
   Let $G=F(X)$ denote the free group on countably many generators. For a given word  $w\in F$, let $|w|$ denote its word-length.

   We  now consider the function $\phi\colon G\to \C$ given by  $\phi(w)=9^{-|w|}$ for every $w\in G$, where we are assuming, as usual, that $|e|=0$.
  This is a  so-called \emph{Haagerup} function and, as any such, is positive definite, \cite[Theorem 1]{michefiga80}.

By    \cite[Theorem 2]{michefiga80},  applied to $\phi^k$ for each  $k\in \N$, $\norm{\phi^k-u}>1$ for every $u\in A(G)$. Hence $(\phi_{[n]})_n$ is not convergent by Corollary \ref{C:umeCaractHso}.

 On the other hand, if we define
  $\psi=\phi-\Cf_{\{e\}}$, then
 \begin{align*}
  \sup_{w\in G}\left|\psi(w)(1+|w|)^2\right|
 &  =  \sup_{n\geq 1} 9^{-n} (n+1)^2<\frac{1}{2}.
\end{align*}
The estimate obtained in     \cite[Corollary 2.4]{branforrzwar17},  a consequence of Haagerup's inequality \cite[Lemma 1.4]{haag79}, shows then that
$ \norm{M_\psi}<1.$
This implies that $\lim (M_{\psi})^n=0$. Noting that, for each $n\in \N$, $\psi^n =\phi^n-\Cf_{e}$, we conclude that $\lim_n ||M_{\phi^n}-M_{\Cf_{e}}\|=0$.


%
\end{proof}

\begin{corollary}
\label{adapted}
Let $G$ be a locally compact group and let $\phi\in P^1(G)$ be adapted. Then the following are equivalent
\begin{itemize}
 \item[(i)] $M_\phi\colon A(G)\to A(G)$ is uniformly mean ergodic.
 \item[(ii)]  $G$ is discrete and $\phi$ is spread-out.
 \item[(iii)] 1 is isolated in $\sigma(M_\phi)$.
 \end{itemize}
\end{corollary}

All the preceding results can be easily  adapted to analyze the convergence of the iterates $M_\phi^n$. We state here the adapted case.
\begin{corollary}
\label{adapted2}
Let $G$ be a  locally compact group and let $\phi\in P^1(G)$ be adapted. Then the following are equivalent
\begin{itemize}
 \item[(i)] $(M_\phi^n)_n$ is norm convergent.
  \item[(ii)] $1$ is isolated in $\sigma(M_\phi)$ and $\sigma(M_\phi)\cap \T=\{1\}$.
 \item[(iii)] $G$ is discrete and  $\phi$ is both spread-out and strictly aperiodic.
 \item[(iv)] $M_\phi:A(G)\to A(G)$ is uniformly mean ergodic and $(M_{\phi^n})_n$ converges  in the strong operator topology.
 \end{itemize}
\end{corollary}
\begin{proof}
Theorem \ref{P:UMEequiv1isol}   shows that assertion (i) implies assertion (ii).
Corollary \ref{adapted} shows that (ii) implies the first two conditions of (iii); strict aperiodicity is deduced directly from $\overline{\phi(G)}\subseteq \sigma(M_\phi)$.

We now check that  (iii) implies (i) and the first three statements will have been shown to be  equivalent.

 Assume now that   assertion (iii) holds, so that $\phi=\phi_a+\phi_s$ with $\phi_a\in A(G)$ and $\norm{\phi_s}<1$.
  Consider then $z\in \sigma(\phi)\cap \T$. There is then $\chi\in \Delta(B(G))$ such that $\chi(\phi)=z$. It follows that   $\restr{\chi }{A(G)}\neq \{0\}$, for, otherwise,  \[|\chi(\phi)|=|\chi(\phi_s)|\leq \norm{\phi_s}<1.\]

    Knowing that there is $x\in G$ such that $\chi(u)=u(x)$ for every $u\in A(G)$, we pick  $u_x\in A(G)$ with $u_x(x)\neq 0$. And using that $\phi u_x\in A(G)$, we see that $\phi(x)=\chi(\phi)$, i.e., that  $x\in E_\phi$. But  this means $z=1$, $\phi $ being strictly aperiodic. We conclude that $\sigma(M_\phi)\cap \T\subseteq \{1\}$. Corollary \ref{adapted} then ensures that we can apply statement (ii) of Theorem \ref{P:UMEequiv1isol} and deduce that $(M_\phi^n)_n$ is norm convergent.

The equivalence of (iii) and (iv) follows from Corollary \ref{adapted} and  Proposition \ref{iteratessot}.
\end{proof}

\begin{remark}\label{dissoc}
It could be tempting to conjecture that, for a strictly aperiodic $\phi\in P^1(G)$, $\sigma(M_\phi)\cap \T=\{1\}$. In that way, the spread-out  property  would take care of uniform ergodicity and strict aperiodicity would take care of $\sigma (M_\phi)\cap \T=\{1\}$. That is, however, false.
Let $\mu$ be a measure supported on an  independent Cantor subset $P$ of $\T$, $\mu$ is then strictly aperiodic. But  powers  of $\mu$  are mutually singular, \cite[Theorem 5.3.2]{rudin90}, and it is known that, in that case $\sigma(\mu)=\T$, \cite[Theorem 6.1.1]{grahmcge79}.
\end{remark}

\subsection{Uniform mean ergodicity and quasi-compact operators}
 The theorem of Yosida and Kakutani (Theorem~\ref{yosidakaku}) states that a quasi-compact operator with bounded powers is always uniformly mean ergodic. On the other hand, if $G$ is discrete, Lau \cite{lau79} characterizes compact multipliers on $A(G)$ as precisely those given by functions of $A(G)$. Hence, for a discrete amenable group and an adapted $\phi\in \mathcal{P}^1(G)$, $M_\phi$   is  quasi-compact    if and only if
 $\phi$ is spread-out.   We observe next that the condition $G$ discrete is required when discussing quasi-compactness.
\begin{lemma}
\label{L:QCimpDisc}
    Let $G$ be a locally compact group and let $\phi\in P^1(G)$ with $H_\phi$ open. If $M_\phi$ is quasi-compact, then $G$ is discrete.
\end{lemma}
\begin{proof}
    Assume $G$ is not discrete. Then, $H_\phi$, being open,  is not discrete either. For each $u\in A(H_\phi)$, we denote by $ \overset{\circ}{\, u }\in A(G)$ the extension of $u$ to $G$, by setting it to  $0$ in $G\setminus H_\phi$. By  \cite[Proposition 2.4.1]{kalau18}, we have $ \norm{\overset{\circ}{\, u}}=\norm{u}_{A(H_\phi)}$.
    Applying  \cite[Theorem 3.2]{chou82} we can find a sequence $(u_n)_n\subseteq A(H_\phi)$ with
   $\|u_n\|_{\mbox{\tiny $A(H_\phi)$}} = 1$   for every $n\in \N$,  and $\|u_n-u_m\|_{A(H_\phi)} = 2$, for $n\neq m$.

    Since $M_\phi$ is quasi-compact, there exist $k\in \N$ and $K$ a compact operator with $\|M_\phi ^k - K\|<A< 1$. Since $K$ is compact we can assume, passing to a subsequence, if necessary,  that $(K\overset{\circ}{ \,u _n})$ is a Cauchy sequence. Then for $n,m$ big enough, $\|K\overset{\circ}{ \,u _n}- K\overset{\circ}{\, u _m}\| <2-2A$. We have then,
    \begin{align*}
        2A &>
          \|M_\phi ^k \overset{\circ}{ \,u _n} - M_\phi ^k \overset{\circ}{ \,u _m}\| - \| K \overset{\circ}{ \,u _n} - K\overset{\circ}{ \,u _m}\|
        \\
        &\geq \|M_\phi ^k \overset{\circ}{ \,u _n} - M_\phi ^k \overset{\circ}{ \,u _m}\| +2A-2
        \\
        &= \|\overset{\circ}{ \,u _n}- \overset{\circ}{ \,u _m}\| +2A -2.
    \end{align*}
    As $ \norm{\overset{\circ}{\, u}}=\norm{u}$ for every $u\in A(H_\phi)$,  , we find that
    \[
        2A> \| u_n - u_m\|_{\mbox{\tiny $A(H_\phi)$}} +2A -2 = 2A,
    \]
    which is a contradiction, so $G$ is discrete.
\end{proof}

One can characterize quasi-compactness of $M_\phi$ in the same way as quasi-compactness of the convolution operator $\lambda_1(\mu)$ was characterized in \cite[Theorem 5.24]{gajoro24}.
Our proof here is slightly simpler.

\begin{theorem}
\label{qc}
    Let $G$ be a locally compact group and $\phi\in P^1(G)$. The following assertions are equivalent:
    \begin{enumerate}
        \item[(1)] The operator $M_\phi$ is quasi-compact.
        \item[(2)] $(M_{\phi_{[n]}})_n$ is norm convergent to a finite dimensional projection.
        \item[(3)] $G$ is discrete, $H_\phi$ finite and $\tilde \phi$ spread-out.
    \end{enumerate}
\end{theorem}
\begin{proof}
    Assertion $(1)$ implies assertion $(2)$ by Yosida-Kakutani Theorem~\ref{yosidakaku}.

    Assume $(2)$ holds, then $\tilde\phi$ is spread-out by Theorem~\ref{T:umeCaractHso}.
     Since $(M_{\phi_{[n]}})_n$ converges to $M_{\Cf_{H_\phi}}$ we get that $\Cf_{H_\phi}\cdot A(G)$ is finite dimensional.
     Now, for any finite family $x_1,\ldots,x_N\in H_\phi$, we can find $U_N$ an open neighbourhood of $e$, such that $x_i U_N \cap x_j U_N =\emptyset$, for $i\neq j$. This implies that $\{\Cf_{H_\phi}\Cf_{x_1U_N},\ldots , \Cf_{H_\phi}\Cf_{x_NU_N}\}$ are linearly independent. As $\Cf_{H_\phi}\cdot A(G)$ is finite dimensional, we see that  $H_\phi$ has to be finite. Since  it is also open (by Theorem \ref{P:meOpen}) $G$ is discrete.

    Finally, when $G$ is discrete, the functions of $A(G)$ define compact multipliers by \cite[Lemma 6.8]{lau79}. Taking into account that $A(G/H_\phi)\subseteq A(G)$ when $H_\phi$ is finite, if  $\tilde\phi$ is spread-out, there are $k\in \N$ and $\phi_a\in A(G)$ such that  $\|M_\phi ^k - M_{\phi_a}\|<1$ and $M_{\phi_a}$ is compact. Therefore assertion  (1) holds.
\end{proof}

\begin{example}
  Uniform convergence of the means of a multiplication  operator $M_\phi$, with $\phi\in B(G)$,  to a finite dimensional operator does not necessarily imply that $M_\phi$ is quasi-compact, if $\phi$ is not positive-definite.\end{example}\begin{proof}
It is indeed enough to consider the constant functions $\phi(x)=\alpha\in \T$, $\alpha\neq 1$, on an infinite group $G$. The sequence
$(M_{\phi_{[n]}})_n$ is then convergent to 0.
The operator $M_\phi$ is not quasi-compact, since if  $\norm{M_\phi^n-K}=\norm{\alpha^n Id-K}<1$, for some $n\in \N$ and some compact operator $K$, then   $K$ would be   compact and invertible, which is impossible unless $A(G)$ is finite dimensional.  This shows that, in Theorem \ref{qc},  (2) does not imply (1)  when $\phi\notin P(G)$.
  \end{proof}

\section{Dualizing random walks}
If $G$ is a locally compact group,  a random walk whose transitions are  determined by a probability measure  $\mu$ is ergodic if and only if the operator $\lambda_1^0(\mu)$,  obtained from restricting the convolution operator $\lambda_1(\mu)$ to the augmentation ideal $L_1^0(G)=\{f\in L_1(G)\colon \int f(x)\dmg(x)=0\}$, is mean ergodic and its means converge to 0, see \cite{rose81}.

In our context, we replace $L_1^0(G)$ by $A_0(G)=\{u\in A(G)\colon u(e)=0\}$ and consider  the operator $M_\phi^0$ that results from restricting $M_\phi$, $\phi \in \mathcal{P}^1(G)$, to $A_0(G)$. In this section, we address the problem of characterizing under which conditions the means of $M_\phi^0$ converge to 0, both in the strong operator topology and the uniform norm.
\subsection{The operator $M_\phi^0$}
The results in this subsection  offer a glance into the nature of the operator $M_\phi^0$.
They follow exactly the same pattern of the results obtained in \cite[Sections 3 and 5]{gajoro24} for the operator $\lambda_1^0(\mu)$.  We actually  need only  prove Theorem \ref{normm0}, as the rest of the proofs can be applied to this case.
\begin{lemma}
  \label{lem:norm}
  Let $G$ be a locally compact group and let $\phi\in B(G)$. If there are $u\in A(G)\cap P^1(G) $ and $\psi\in P^1(G)$ such that
  \begin{enumerate}
    \item $\norm{\phi u-\psi}\geq M$  for some $M>0$ and
    \item $\norm{\phi u \psi-\psi}\leq \varepsilon$, for some $0<\varepsilon<M$, then
  \end{enumerate}
  \[\norm{M^0_\phi}\geq \frac{M-\varepsilon}{2}.
  \]
\end{lemma}
    \begin{proof}
      Since $u-\psi u \in A_0(G)$,
      \[ \norm{M_\phi^0}\geq \frac{\norm{(\phi-\phi \psi)u}}{\norm{u-\psi u}}.\]
      To achieve the estimate of this lemma one just needs to observe that
      \[\norm{(\phi-\phi \psi)u}\geq \norm{\phi u-\psi}-\norm{\phi u \psi-\psi}\]
and
      \[ \norm{u-\psi u}\leq 2.\]
    \end{proof}

For our next proof, we need the concept of TI-net.
A net $(u_\alpha)_\alpha \subseteq A(G)\cap P^1(G)$ is a TI-net if
    \[
       \lim_\alpha \norm{uu_\alpha -u_\alpha}= 0, \quad \mbox{for any }  u\in A(G).
    \]
The existence of TI-nets in nondiscrete groups was shown  in \cite[Proposition 3]{rena72}, see also \cite{chou82,filagali22}.
\begin{theorem}
\label{normm0}
If $G$ is a nondiscrete locally compact group and $\phi \in P^1(G)$, then $\norm{M_\phi^0}=1$.
\end{theorem}
\begin{proof}
  Let  $(u_\alpha)_\alpha$ be a TI-net in $A(G)\cap P^1(G)$ and let $u_0$ be any element  of $A(G)\cap P^1(G)$. By  \cite[Lemma 3.1]{chou82},
$\lim_\alpha \norm{\phi  u_0 -u_\alpha}=2$, while $\lim_\alpha \norm{ \phi u_0 u_\alpha-u_\alpha}=0$. Lemma \ref{lem:norm} applied,  for each $\alpha$, to $u=u_0$ and $\psi=u_\alpha$ yields that
$\norm{M_\phi^0}=1$.
\end{proof}
\begin{corollary}
  Let $G$ be a locally compact group and let $\phi\in P^1(G)$. If $\phi $ is uniformly ergodic, then $G$ is discrete.
\end{corollary}
We do not know if the spectra of $M_\phi^0$  and $M_\phi$ are the same but they are definitely related. This is explored in the following theorem. It can be proved exactly as   Proposition 5.8 and Corollary 5.9 of \cite{gajoro24}, we only have to replace   [Proposition 3.2, loc. cit.] by Theorem \ref{normm0} here. Recall that  a complex number $z$ is in $\sigma_{\mathrm{ap}}(T)$, the \emph{approximate spectrum} of an operator  $T\in \ll(E)$, $E$ a Banach space,   if there exists a sequence $x_n$ in the unit sphere of $E$ such that $\lim_n \norm{Tx_n-zx_n}=0$.
\begin{theorem}
  \label{thm:specm0}
  Let $G$ be a locally compact group and let $\phi\in P^1(G)$. Then:
  \begin{enumerate}
    \item If $G$ is not discrete, then $ \sigma_{\mathrm{ap}}(M_\phi^0)= \sigma_{\mathrm{ap}}(M_\phi)$.
    \item 1 is isolated in $\sigma(M_\phi)$ if and only if 1 is isolated in $\sigma(M_\phi^0)$.
  \end{enumerate}
\end{theorem}

\subsection{Ergodicity of $\phi$ }
In analogy with the convolution case, we have defined a function $\phi\in P^1(G)$ to be ergodic when $\phi_{[n]}u$ converges to 0 for every $u\in A_0(G)$. We prove in this subsection that $\phi$ is ergodic if and only if $\phi$ is adapted. For  convolution operators the situation is more involved. In the case of commutative or compact groups the conclusion is the same, a measure is ergodic if and only if it is adapted. However, when  the group is not amenable, no measure  can be ergodic \cite{rose81} and  every group that is finitely generated  and solvable, but  is not virtually nilpotent,   admits a nonergodic adapted measure  \cite{frischetal19}.

The project of characterizing ergodicity in $P^1(G)$ has been taken up before, often in somewhat more general contexts. Kaniuth, Lau and \"Ulger \cite[Theorem 3.4]{kanilauulger10} work with multipliers on quite general Banach algebras, albeit requiring them to have bounded approximate identities. This condition is avoided in Theorem 5.1.1 of Guex's Ph. D. dissertation  \cite{guex} but this theorem is not correct as stated.  In \cite[Theorem 3.3]{laulose99} Lau and Losert consider so-called \emph{ strongly ergodic sequences}. Their definition of strong ergodicity would then mean that  $(\pi(\phi_{[n]}))_n$ is convergent, for every representation $\pi \colon A(G)\to \ll(\h)$ of $A(G)$ as operators on a Hilbert space $\h$. This requires the existence of a canonical way of extending representations from $A(G)$ to representations of $B(G)$, and for that, again, they have to restrict their characterization to amenable groups.

We provide  next  a short direct characterization of ergodicity in $P^1(G)$ that subsumes the ones mentioned in the previous paragraph. Our extension is a clean analog of the measure-theoretic concept and does not require amenability.
%


The following elementary fact will smooth our proof of Theorem \ref{T:EiffA}.
\begin{lemma}\label{M0}
Let $G$ be a locally compact group and let $\phi \in P^1(G)$.  If $\phi$ is adapted and $M_\phi^0$ is (uniformly) mean ergodic, then $\phi$ is (uniformly) ergodic.
\end{lemma}
\begin{proof}
  Let $P$ be the projection in $\ll(A_0(G))$ such that $\lim_n M_{\phi_{[n]}}^0=P$ in the corresponding topology and let $u\in A_0(G)$. If $e\neq x$, then $(Pu)(x)=\lim_n \phi_{[n]}(x)u(x)=0$. As $Pu\in A_0(G)$, this means that $Pu=0$, and hence that $P=0$.
\end{proof}

For our characterization, we need to recall the concept of support of  an element of $\VN(G)$, which  is based on the module action of $A(G)$ on $ \VN(G)$, described in page \pageref{module}, with $A(G)$ playing the role of $\mA$.
 The support of $T$, denoted $\supp T$, is then defined as the  set of all points $a\in G$ satisfying that $\lambda_2(\delta_a)$ is the weak$^*$-limit of operators of the form $ T\cdot u$, $u\in A(G)$. See \cite{eyma64} or \cite[Section 2.5]{kalau18} for all this.

\begin{theorem}
\label{T:EiffA}
Let $G$ be a locally compact group and let $\phi \in P^1(G)$. Then $\phi$  is ergodic if and only it is adapted.
\end{theorem}
\begin{proof}
If $\phi$ is ergodic, then  it must also be adapted, else there would exist $e\neq x\in H_\phi$ and $u\in A_0(G)$ with $u(x)\neq 0$, so $|\phi_{[n]}(x) u (x)|=|u(x)|\neq 0$, for all $n\in\N$, contradicting ergodicity.

 For the converse, assume that $\phi$ is adapted and suppose that $\phi$ is not ergodic. Then $M_\phi^0$ cannot be  mean ergodic either, by the preceding  Lemma, and the ergodic decomposition  is not satisfied. Since $\ker (I-M_\phi ^0) =\{0\}$,  $\phi$ being adapted, this means that $A_0(G)\neq \overline{\{(I-M_\phi ^0)(u)\colon u\in A_0(G)\}}$. There exists then $T\in \VN(G)$ with $\restr{T}{A_0(G)}\neq 0$ such that 
 $\langle T,u\rangle=\langle T, \phi u\rangle$, for every $u\in A_0(G)$.

    Let $x\in \supp (T)$. Then  $\lambda_2(\delta_x)=w^*-\lim_\alpha  T\cdot u_\alpha$, for some net $u_\alpha\in A(G)$. If $x\neq e$, We can take $u\in A_0(G)$ with $u(x)\neq 0$ and
 \begin{align*}
        u(x)= \lim_\alpha \langle  T\cdot u_\alpha, u\rangle&= \lim_\alpha \<{T,u_\alpha u}\\&=
        \lim_\alpha \<{T,u_\alpha\phi  u}\\&=
         \lim_\alpha \langle T\cdot u_\alpha, \phi u\rangle = \phi (x) u(x).
 \end{align*}
    Therefore, $\phi(x)=1$, but $\phi$ is adapted, so $x=e$.

      So, $\supp(T)=\{e\}$ and \cite[Corollary 2.5.9]{kalau18} proves that $T$ is a multiple of $\lambda_2(\delta_e)$ and hence that $\restr{T}{A_0(G)}\equiv 0$, a contradiction.
      \end{proof}

The picture on ergodicity of $\phi\in \mathcal{P}^1(G)$ is completed by the solution to the complete mixing problem that follows from Theorem 2.1 of  \cite{kanilauulger10}. This problem is still open for convolution operators, see \cite[Remark 2.10]{gajoro24}.

\begin{theorem}[Kaniuth, Lau and \"Ulger]
    Let $\phi\in P^1(G)$, then $\phi$ is strictly aperiodic if, and only if, it is completely mixing.
\end{theorem}

    %
    %

\begin{remark}The results of this section show that the mean ergodic properties of the operators $M_\phi$ and $M_{\phi}^0$ are quite different. We see in the next subsection   that the situation changes drastically when we study uniformly ergodic behaviour.
\end{remark}

\subsection{Uniform ergodicity}
After the work already done in this and previous sections, uniform ergodicity can be characterized without effort.

  We remark that all conditions in the statement of the following Theorem  imply that $G$ is discrete,   for,  when $\phi$ is adapted  and  1 is isolated in $\overline{\phi(G)}$,   $H_\phi$ must reduce to  $\{e\}$ and be open.

\begin{theorem}
\label{ue}
    Let $G$ be a  discrete group and let $\phi\in P^1(G)$. The following assertions are equivalent:
    \begin{enumerate}
    \item $\phi$ is uniformly ergodic.
     \item $\phi$ is adapted  and $(M_\phi)_{[n]}$ is uniformly mean ergodic.
        \item $\phi$ is adapted and spread-out.
        \item $\phi$ is adapted  and $M_\phi$ is quasi-compact.
        \item $\phi$ is adapted  and $1$ is isolated in $\sigma(M_\phi)$.
        \item $\phi$ is adapted  and $1$ is isolated in $\sigma(M_\phi ^0)$.
    \end{enumerate}
\end{theorem}

\begin{proof}
By Theorem \ref{T:EiffA}, adaptedness is a necessary condition for uniform ergodicity, its appearance in items (2)--(6) needs not further mention.

  Assertion (1) implies (2) because $M_\phi^0$ is just the restriction of $M_\phi$ to the hyperplane $A_0(G)$ and uniform mean ergodicity  of such a restriction implies uniform mean ergodicity  of the operator, see \cite[Proposition 4.4]{gajoro24}.  The converse follows from Lemma \ref{M0}.

Corollary \ref{adapted} and Theorem \ref{qc} prove that assertions (2), (3), (4) and (5) are equivalent.

Finally, Theorem \ref{thm:specm0} proves that assertion (5) and  (6) are equivalent.


\end{proof}

The same approach shows that the  uniform completely mixing problem can be solved combining  Corollary~\ref{adapted2} With Theorem~\ref{ue}.



\begin{theorem}
\label{ucm}
   Let $G$ be a locally compact group and $\phi\in P^1(G)$.  Then $\phi$ is uniformly completely mixing if and only if  it is strictly aperiodic and uniformly ergodic.
\end{theorem}

%
%

We finish giving the results that we get if we proceed similarly but using  Corollary \ref{C:umeCaractHso} instead of Theorem \ref{T:umeCaractHso}. When $G$ is amenable, this is nothing but putting together Theorem \ref{ue} and Theorem \ref{ucm} above. Example \ref{ex} shows that when $G$ is not amenable, the situation differs.
\begin{theorem}
Let $G$ be an amenable discrete locally compact group and let $\phi\in P_1(G)$ be adapted. Consider the following conditions:
\begin{itemize}
\item[(i)] $\lim_n\norm{\phi_{[n]}-\Cf_e}=0$.
\item[(ii)] There is $k\in\N$ and  $u\in A(G)$ such that $\norm{\phi^k-u}<1$.
\item[(iii)] 1 is isolated in $\sigma(\phi)$.
\item[(a)] $\lim_n\norm{\phi^n-\Cf_e}=0$.
\item[(b)] 1 is isolated in $\sigma(\phi)$ and $\phi$ is strictly aperiodic.

\end{itemize}
Conditions (i), (ii) and (iii) are equivalent. Conditions (a) and (b) are equivalent.
\end{theorem}
\begin{proof}
The proof of the equivalence between (i) and (ii) is completely analogous to the proof of Theorem \ref{ucm}, using Corollary \ref{C:umeCaractHso}. Using Proposition \ref{P:UMEequiv1isol} we get that both are equivalent to (iii). It is trivial that (a) implies (b). If we assume (b), then (iii) holds and then also (ii), which yields $\|\widetilde M^k _\phi - \widetilde M_u\|<1$. Note that, since $u\in A(G)$, $\widetilde M_u$ is the limit of finite range operators, which implies that $\widetilde M_\phi$ is quasicompact. We get (b) by Theorem \ref{yosidakaku}, since $\sigma_p(\widetilde{M}_\phi)=\overline{\phi(G)}$.
\end{proof}


\subsection*{Acknowledgements}
    The second author was partially supported by the project PID2020-119457GB-100 funded by MCIN/AEI/10.13039/501100011033 and “ERDF A way of making Europe” and by CIAICO/2023/242 funded by GVA.

    The third author is supported by Ayudas Margarita Salas 2021-2023 of Universitat Politecnica de Valencia funded by the Spanish Ministry of Universities (Plan de Recuperacion, Transformacion y Resiliencia) and European Union-Next generation EU (RD 289/2021 and UNI/551/2021) .

\def\cprime{$'$} \def\cprime{$'$} \def\cprime{$'$}
  \def\polhk#1{\setbox0=\hbox{#1}{\ooalign{\hidewidth
  \lower1.5ex\hbox{`}\hidewidth\crcr\unhbox0}}}
  \def\polhk#1{\setbox0=\hbox{#1}{\ooalign{\hidewidth
  \lower1.5ex\hbox{`}\hidewidth\crcr\unhbox0}}} \def\cprime{$'$}

\end{document}